\documentclass[11pt,reqno]{amsart}
\usepackage{graphicx}
\usepackage{float}
\usepackage[caption = false]{subfig}
\usepackage{enumerate}
\usepackage{mathtools}
\usepackage{breqn}
\usepackage{array, geometry, graphicx}
\usepackage{amsmath,amsfonts,paralist,amssymb,amsthm,mathrsfs}
\usepackage{pdfpages}
\newtheorem{example}{Example}
\newtheorem{theorem}{Theorem}[section]
\newtheorem{corollary}[theorem]{Corollary}
\newtheorem{lemma}[theorem]{Lemma}

\theoremstyle{definition}
\newtheorem{definition}[theorem]{Definition}
\theoremstyle{remark}
\newtheorem{remark}[theorem]{Remark}
\numberwithin{equation}{section}

\DeclareMathOperator{\sech}{sech}

\begin{document}
	
\title[Hyperbolic Cosine function]{On a subfamily of starlike functions related to Hyperbolic Cosine function}
	\thanks{The first author is supported by Delhi Technological University, New Delhi}
	\author[Mridula Mundalia]{Mridula Mundalia}
	\address{Department of Applied Mathematics, Delhi Technological University, Delhi--110042, India}
	\email{mridulamundalia@yahoo.co.in}
	\author{S. Sivaprasad Kumar}
	\address{Department of Applied Mathematics, Delhi Technological University, Delhi--110042, India}
	\email{spkumar@dce.ac.in}

	\subjclass[2010]{30C45, 30C80}
	
	\keywords{Univalent functions, Starlike functions, Radius problems, Hyperbolic Cosine function, Subordination}
\begin{abstract}
We introduce and study a new Ma-Minda subclass of starlike functions $\mathcal{S}^*_{\varrho},$ defined as
\[\mathcal{S}^{*}_{\varrho}:=\left\{f\in\mathcal{A}:\frac{zf'(z)}{f(z)} \prec \cosh \sqrt{z}=:\varrho(z), z\in\mathbb{D} \right\},\] associated with an analytic univalent function $\cosh \sqrt{z},$ where we choose the branch of the square root function so that $\cosh\sqrt{z}=1+z/2!+z^{2}/{4!}+\cdots.$ We establish certain inclusion relations for $\mathcal{S}^{*}_{\varrho}$ and deduce sharp $\mathcal{S}^{*}_{\varrho}-$radii for certain subclasses of analytic functions.
\end{abstract}
	
	\maketitle
	
	\section{Introduction}

 Let $\mathcal{A}_{n}$ be the class of all analytic functions defined on the open unit disc $\mathbb{D}:=\left\{z\in\mathbb{C}:|z|<1\right\},$ with Taylor series representation of the form $f(z)=z+a_{n+1}z^{n+1}+a_{n+2}z^{n+2}+\cdots.$ Let $\mathcal{A} := \mathcal{A}_{1}.$ Assume $\mathcal{S}\subset \mathcal{A}$ as the class of univalent functions.
If $f(z)$ and $g(z)$ are analytic functions in $\mathbb{D},$ then $f(z)$ is said to be subordinate to $g(z)$ $(f \prec g),$  if there exists a self-map $w(z)$ on $\mathbb{D}$  such that $w(0)=0$ and $f(z)=g(w(z)).$ For instance,  if $g(z)$ is a univalent function in $\mathbb{D}$, then $f\prec g$ if and only if $f(0)=g(0)$ and $f(\mathbb{D})\subset g(\mathbb{D}).$
In 1992, Ma and Minda \cite{Ma & Minda} investigated the following  subclasses of $\mathcal{A}$ using the notion of subordination
 \[\mathcal{S}^{*}(\phi)=\left\{f\in\mathcal{A}:\frac{zf'(z)}{f(z)}\prec \phi(z), z\in\mathbb{D}\right\}\]
and \[\mathcal{C}(\phi):=\left\{f\in\mathcal{A}:1+\frac{zf''(z)}{f'(z)} \prec \phi(z), z\in\mathbb{D}\right\}.\]
In the above defined classes, the expressions $zf'(z)/f(z)$ and $1+z''(z)/f'(z)$ are subordinate to an analytic univalent function $\phi(z)$ such that $\phi'(0)>0$ and $\operatorname{Re} \phi(z)>0$ $(z\in\mathbb{D}).$ Furthermore, $\phi(z)$ is symmetric about the real axis and starlike with respect to $\phi(0)=1$. Several authors have previously handled the Ma and Minda classes for various choices of $\phi(z),$ some are enlisted below  

\begin{table}[ht]
\renewcommand{\arraystretch}{1.17}
  \caption{Ma-Minda starlike classes for special choices of $\phi(z)$} 
  \centering 
  \begin{tabular}{llll} 
  \hline 
    {\bf{Class $\mathcal{S}^{*}(\phi)$}}  & \textbf{$\phi(z)$}  & \textbf{References}     \\ [0.9ex] 
    \hline  
           $\mathcal{S}^{*}_{e}$       &   $e^{z}$ &   \cite{Mendiratta n Nagpal} Mendiratta et al.  \\ 
     $\mathcal{S}^{*}_{L}$    &   $\sqrt{1+z}$   &  \cite{Sokol n Stankiewicz} Sok\'{o}\l \ et al.  \\ 
       $\mathcal{S}^*_{q}$                &  $z+\sqrt{1+z^2}$  &    \cite{Raina} Raina et al. \\  
       $\mathcal{S}^{*}_{s}$       &   $(1+sz)^{2},$ $-1 \leq s \leq 1$ &    \cite{Masih n Kanas} Masih et al.  \\ 
$\mathcal{S}^{*}(q_{\kappa})$    &    $\sqrt{1+\kappa z},$ $0<\kappa \leq 1$  &  \cite{Aouf n Sokol} Sok\'{o}\l \ et al. \\ 
             $\mathcal{SS}^{*}(\beta)$       &   $((1+z)/(1-z))^{\beta},$ $0<\beta\leq 1$   &    \cite{Stankiewicz} Stankiewicz  \\ 
             $\mathcal{S}^{*}[A,B]$   &  $(1+Az)/(1+Bz)$   &\cite{Janowski}  W. Janowski \\ 
    $\mathcal{S}^{*}(\beta)$     &   $(1+(1-2\beta)z)/(1-z),$ $0\leq\beta< 1$  &   \cite{Robertson(1936)} Robertson \\  
        \hline 
\end{tabular}
\label{EQNTable2}
\end{table}
\noindent The classes $\mathcal{S}^{*}_{e},\mathcal{S}^{*}_{s},$ $\mathcal{S}^{*}(q_{\kappa}),$ $\mathcal{S}^{*}[A,B]$ and $\mathcal{SS}^{*}(\beta)$  were widely studied in \cite{Aouf n Sokol, Janowski, Masih n Kanas, Mendiratta n Nagpal, Hussain & Darus}. For instance, a number of sufficient conditions in terms of coefficient estimates for the class $\mathcal{SS}^{*}(\beta)$ are studied in \cite{Nezhmetdinov & Ponnusamy(2005)} and references therein.

For the present study we examine the function $\varrho_{\sigma}(z):= \cosh \sigma \sqrt{z},$ where $\sigma\in[-\pi/2,\pi/2]-\{0\}$ and we choose the branch of the square root function so that $\cosh\sigma\sqrt{z}=1+\sigma^{2}z/2!+\sigma^{4}z^{2}/{4!}+\cdots.$ Note that $\varrho_{\sigma}(z)$ is an analytic univalent function with $\operatorname{Re}\varrho_{\sigma}(z)>0$ and maps $\mathbb{D}$ onto a convex region. Further it is symmetric about real axis (i.e $\overline{\varrho_{\sigma}(z)}=\varrho_{\sigma}(\overline{z})$) such that $\varrho_{\sigma}'(0)=\sigma^{2}/2>0.$ Consequently,  $\varrho_{\sigma}(z)$ is a Ma-Minda type function. In the recent years, cosine and cosine hyperbolic functions have been investigated, see \cite{Hussain,Raza}. Note that $\varrho(z)=\cosh \sqrt{z},$ $\phi_{1}(z)=\cos z$ and $\phi_{2}(z)=\cosh z$ have identical images, however $\phi_{1}(z)$ $(\phi_{1}'(0)<0)$ and $\phi_{2}(z)$ are non-univalent functions in $\mathbb{D},$ whereas $\varrho(z)$ $(\varrho'(0)=1/2>0)$ is univalent in $\mathbb{D}.$  
Thus the geometry of $\varrho_{\sigma}(z)$ piqued our interest in formulating the following definition, by means of subordination.

\begin{definition}
Let  $\mathcal{S}^{*}_{\varrho_{\sigma}}$ be the class of normalized starlike functions, defined as follows:
\[\mathcal{S}^{*}_{\varrho_{\sigma}}:=\left\{f\in\mathcal{A}:\frac{zf'(z)}{f(z)} \prec \varrho_{\sigma}(z):=\cosh \sigma \sqrt{z}, z\in\mathbb{D}\right\} \quad (\sigma\in[-\pi/2,\pi/2]-\{0\}),\]
where we choose the branch of the square root function so that \[\cosh\sigma\sqrt{z}=1+\frac{\sigma ^2 z}{2!}+\frac{\sigma ^4 z^2}{4!}+\frac{\sigma ^6 z^3}{6!}+\cdots.\]
\end{definition} 
The conformal mapping $\varrho_{\sigma}:\mathbb{D}\rightarrow \mathbb{C},$ maps the unit disc $\mathbb{D}$ onto the region \[\Omega_{\varrho_{\sigma}}:=\{u\in\mathbb{C}:|\log(u+\sqrt{u^{2}-1})|^{2}<\sigma^{2}\}  \quad (\sigma\in[-\pi/2,\pi/2]-\{0\}),\] defined on the principle branch of logarithm and square root functions. For each $\sigma\leq \hat{\sigma},$ observe that $\varrho_{\sigma}(\mathbb{D}) \subset \varrho_{\hat{\sigma}}(\mathbb{D}).$
Moreover, for each circle $|z|=r<1,$ 
\begin{align} \label{e21}
 \left\{
    \begin{array}{ll}
         \displaystyle{\min_{|z|=r}{\operatorname{Re}}\varrho_{\sigma}(z)}=\min_{|z|=r}|\varrho_{\sigma}(z)|=\varrho_{\sigma}(\sqrt{-r})  \\ 
        \displaystyle {\max_{|z|=r} {\operatorname{Re}}\varrho_{\sigma}(z)}=\max_{|z|=r} |\varrho_{\sigma}(z)|=\varrho_{\sigma}(\sqrt{r}).
    \end{array}
\right. \end{align}
Assume $\varrho_{1}(z)=:\varrho(z),$ therefore we have $\mathcal{S}^{*}_{\varrho}=\mathcal{S}^{*}_{\varrho_{1}}.$ In the  present investigation we shall restrict our major workings to a subclass of starlike functions, namely $\mathcal{S}^{*}_{\varrho},$ and deduce radii constants along with some inclusion relations.
In terms of integral representation, we have $f\in\mathcal{S}^{*}_{\varrho}$ if and only if 
\begin{equation}\label{e20}
f(z)=z \exp\left(\int_{0}^{z} \frac{\hat{\varrho}(t)-1}{t} dt\right)
\end{equation}
where $\hat{\varrho}(z) \prec \varrho(z).$ Note that if $\psi_{\hat{ \varrho}}(z)=1+z/3+z^{2}/18$ and  $\phi_{\hat{ \varrho}}(z)=1+\sin\left(z/3\right),$ then evidently $\psi_{\hat{ \varrho}}(z)$ and $\phi_{\hat{ \varrho}}(z)$ are subordinate to $\varrho(z),$ so the corresponding functions  
 \[f_{1}(z)=z \exp{\left({\frac{z}{3}+\frac{z^{2}}{36}}\right)} \quad \text{and} \quad f_{2}(z)=z e^{Si(z)}, \text{where } Si(z)=\int_{0}^{z}\frac{\sin t}{t} dt\] lie in $\mathcal{S}^{*}_{\varrho}.$ Now using the representation in \eqref{e20}, we obtain  different functions, those work as extremal functions for various results. For instance, $\varphi_{\varrho _{n}}\in \mathcal{A}$ $(n=2,3,4,\ldots),$ defined as 
\begin{align}\label{e22}
\varphi_{\varrho _{n}}(z)=z \exp\left(\int_{0}^{z} \frac{\varrho(t^{n-1})-1}{t} dt\right)=z+\frac{z^n}{2 (n-1)}+\frac{z^{2 n-1}}{48 (n-1)}+\cdots, \end{align} belongs to $\mathcal{S}^{*}_{\varrho}.$ We denote $\varphi_{\varrho}:=\varphi_{\varrho_{2}}.$ For completeness of our class $\mathcal{S}^{*}_{\varrho},$ we give below a remark using the results of \cite{Ma & Minda, Mundalia & Sivaprasad(2020)}.  
  \begin{remark}
   	 For  $f\in\mathcal{S}^{*}_{\varrho}$ and $\varphi_{\varrho}(z)$ be as defined in \eqref{e22}, then for $|z|=r_{0}<1,$ we have 
   	 \begin{enumerate}[(i)]
   	 	\item  $-\varphi_{\varrho}(-r_{0})\leq |f(z)| \leq \varphi_{\varrho}(r_{0})$ (Growth Theorem).
   	 	\item   $\varphi_{\varrho}'(-r_{0}) \leq |f'(z)| \leq \varphi_{\varrho}'(r_{0})$ (Distortion Theorem).
   	 	\item   $|\arg(f(z)/z)|\leq \displaystyle{\max_{|z|=r_{0}}} \arg \left(\varphi_{\varrho}(z)/z\right)$ (Rotation Theorem) .\end{enumerate}
 	Equality for (i)-(iii) holds for some $z_{0}\neq 0$ if and only if $f(z)$ is a rotation of $\varphi_{\varrho}(z).$  Infact if $f\in\mathcal{S}^{*}_{\varrho}$ then either $f$ is a rotation of $\varphi_{\varrho}(z)$ or $f(\mathbb{D})\supset\{v:|v|\leq -\varphi_{\varrho}(-1) \approx 0.619\ldots \}.$
 \end{remark}
Further, from the results in \cite{Mundalia & Sivaprasad(2020)} for each $f\in\mathcal{S}^{*}_{\varrho},$  
   $(i)$ $|a_{2}| \leq 1/2,$ $(ii)$ $ |a_{3}| \leq 1/4,$ $(iii)$ $|a_{4}|\leq 1/6$ and $(iv)$ for any complex constant $\mu,$ $|a_{3}-\mu a_{2}^{2}|\leq \frac{1}{4}\max\{1,|\mu - 7/12|\}.$ These estimates are sharp.
     Equality in $(i)$ holds for the function $\varphi_{\varrho}(z)$ and  $\tilde{f}(z)=z+z^{3}/4$ is an extremal function for $(ii)$ and $(iv).$

\section{Properties of Hyperbolic Cosine function}

We begin with a Lemma which demonstrates a maximal disc  centered at a point $(c,0)$ on the real line, that can be subscribed within $\varrho_{\sigma}{(\mathbb{D})}.$ 

\begin{lemma}\label{l5}
Suppose $\sigma \neq 0,$ then $\varrho_{\sigma}(z)$    satisfies the following inclusion 
\[\left\{u\in \mathbb{C}:|u-c|<r_{\sigma c}\right\}\subset \varrho_{\sigma}(\mathbb{D})=:\Omega_{\varrho_{\sigma}} \quad (-\pi/2 \leq \sigma \leq \pi/2),\]
where 
\begin{align*}
r_{\sigma c}=\left\{\begin{array}{cl}
  c - \cos \sigma, & \cos \sigma < c \leq (\cosh \sigma + \cos \sigma)/2\\
  \cosh \sigma - c, & (\cosh \sigma + \cos \sigma)/2 \leq c < \cosh \sigma.
\end{array}\right.
\end{align*}
\end{lemma}

\begin{proof}
Let  $\Gamma:= \varrho_{\sigma}(e^{i t}),$  $-\pi \leq t \leq \pi$ be the boundary curve of the function $\varrho_{\sigma}(z).$ Due to symmetricity of the curve $\Gamma$ about real-axis, it is enough to consider $ 0\leq t \leq \pi.$ Define a function $G_{c}(\tau)$ as follows:
\[G_{c}(\tau):=\left(c-\cosh\left(\sigma( \cos \tau) \right) \cos\left(\sigma (\sin \tau)\right)\right)^{2}+\sinh^{2}\left(\sigma (\cos \tau) \right)\sin^{2}\left( \sigma (\sin \tau) \right),\] 
where $\tau=t/2.$
Observe that $G_{c}(\tau)$ (see Fig. \ref{fig1} for different values of $c$) is the square of the distance from point $(c,0)$ to $\Gamma.$ Now we study the following cases:\\
\textbf{Case 1:} For $\cos \sigma < c \leq 1,$ $G_{c}(\tau)$ is  monotonically decreasing on $[0,\pi/2],$ then \[r_{\sigma c}=\displaystyle {\min_{\tau\in[0,\pi/2]}}\sqrt{G_{c}(\tau)} = \sqrt{G_{c}(\pi/2)} = c - \cos \sigma.\]  
\textbf{Case 2:} When $1\leq c \leq \sigma_{0},$ where $\sigma_{0}<(\cosh \sigma +\cos \sigma)/2$ is a point at which $G_{c}(\tau)$ changes its character  i.e  $G_{c}(\tau)$ is  monotonically decreasing for $1\leq c \leq \sigma_{0}$ and has three critical points $\{0,\tau_{\tilde{c}},\pi/2\}$ for $\sigma_{0}<c\leq (\cosh \sigma +\cos \sigma)/2,$ where $\tau_{\tilde{c}}\in (0,\pi/2)$ is the only root of the equation 
\begin{align}\label{e12}
&2 c \tan \tau \cos (\sigma \sin \tau) \sinh (\sigma \cos \tau)+2 c \sin (\sigma \sin \tau) \cosh (\sigma \cos \tau) \nonumber\\&= \sin (2 \sigma \sin \tau)+\tan \tau \sinh (2 \sigma \cos \tau).
\end{align} 
Note that $\tau_{c}<\tau_{\tilde{c}}$ whenever $c<{\tilde{c}}.$ Further  \[G_{c}(0)-G_{c}(\pi/2)= (\cosh \sigma - \cos \sigma) (\cos \sigma +\cosh \sigma -2 c) \geq 0.\] Therefore this yields \[r_{\sigma c}=\displaystyle{\min _{\tau\in[0,\pi/2]} \left\{\sqrt{G_{c}(0)},\sqrt{G_{c}(\tau_{\tilde{c}})},\sqrt{G_{c}(\pi/2)}\right\}}  = \sqrt{G_{c}(\pi/2)} = c - \cos \sigma.\] 
\textbf{Case 3:} For $(\cosh \sigma + \cos \sigma)/2 \leq c \leq \sigma_{1},$  where $\sigma_{1}<\cosh \sigma$ is a point at which $G_{c}(\tau)$ changes its character i.e $G_{c}(\tau)$ has three critical points $\{0,\tau_{\hat{c}},\pi/2\},$ where $\tau_{\hat{c}}\in(0,\pi/2)$ is the only root of equation \eqref{e12} and $G_{c}(\tau)$ is an increasing function for $\sigma_{1}< \sigma <\cosh \sigma.$ Infact $G_{c}(0)\leq G_{c}(\pi/2).$ Therefore  \[r_{\sigma c}=\displaystyle\min _{\tau\in[0,\pi/2]}\left\{{\sqrt{G_{c}(0)},\sqrt{G_{c}(\tau_{\hat{c}})},\sqrt{G_{c}(\pi/2)}}\right\}=\sqrt{G_{c}(0)} = \cosh \sigma - c.\]
Hence the result follows.
\end{proof}

\begin{figure}
    \centering
   \subfloat{\includegraphics[width=5cm]{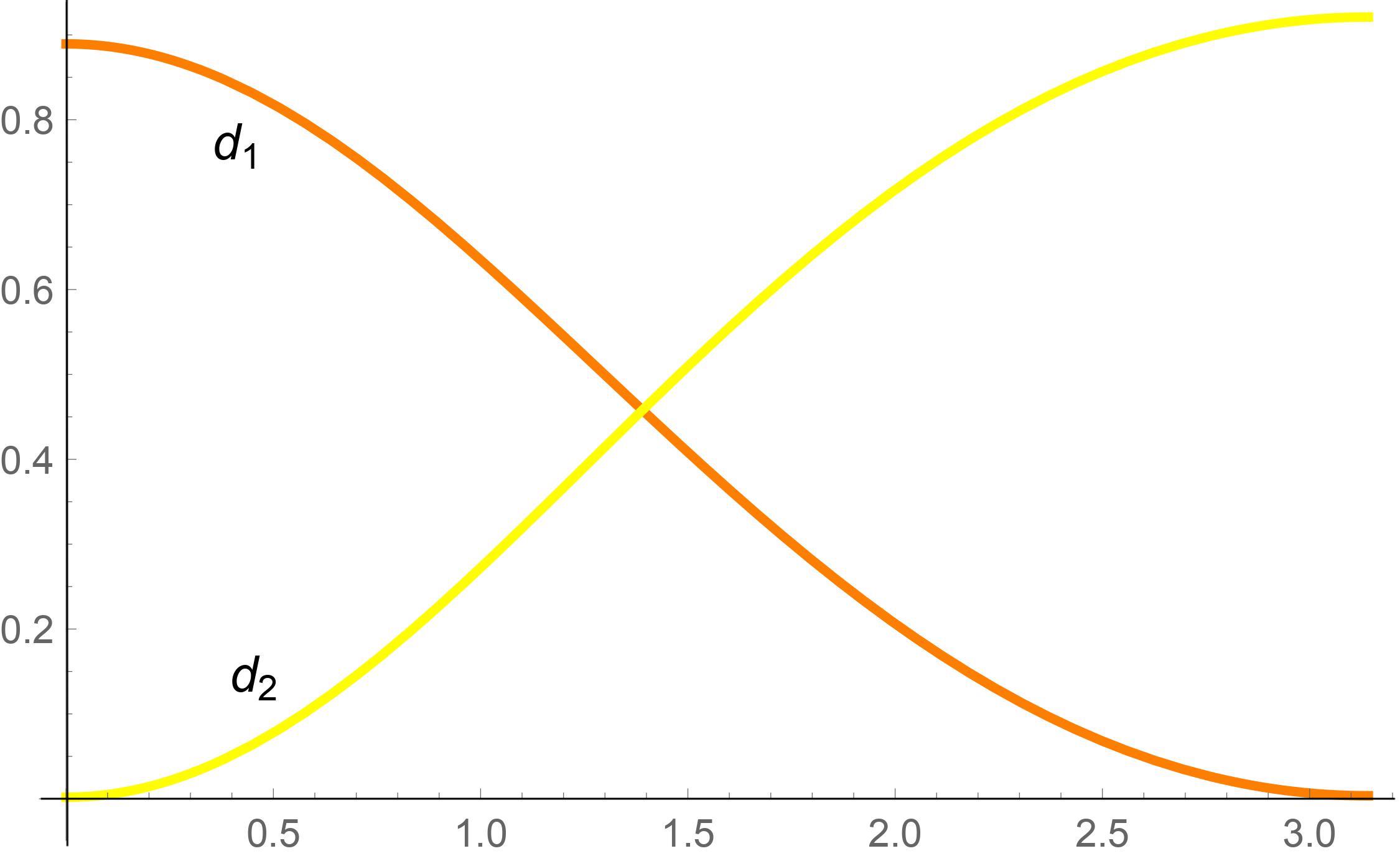}}
   \qquad
    \subfloat{\includegraphics[width=5cm]{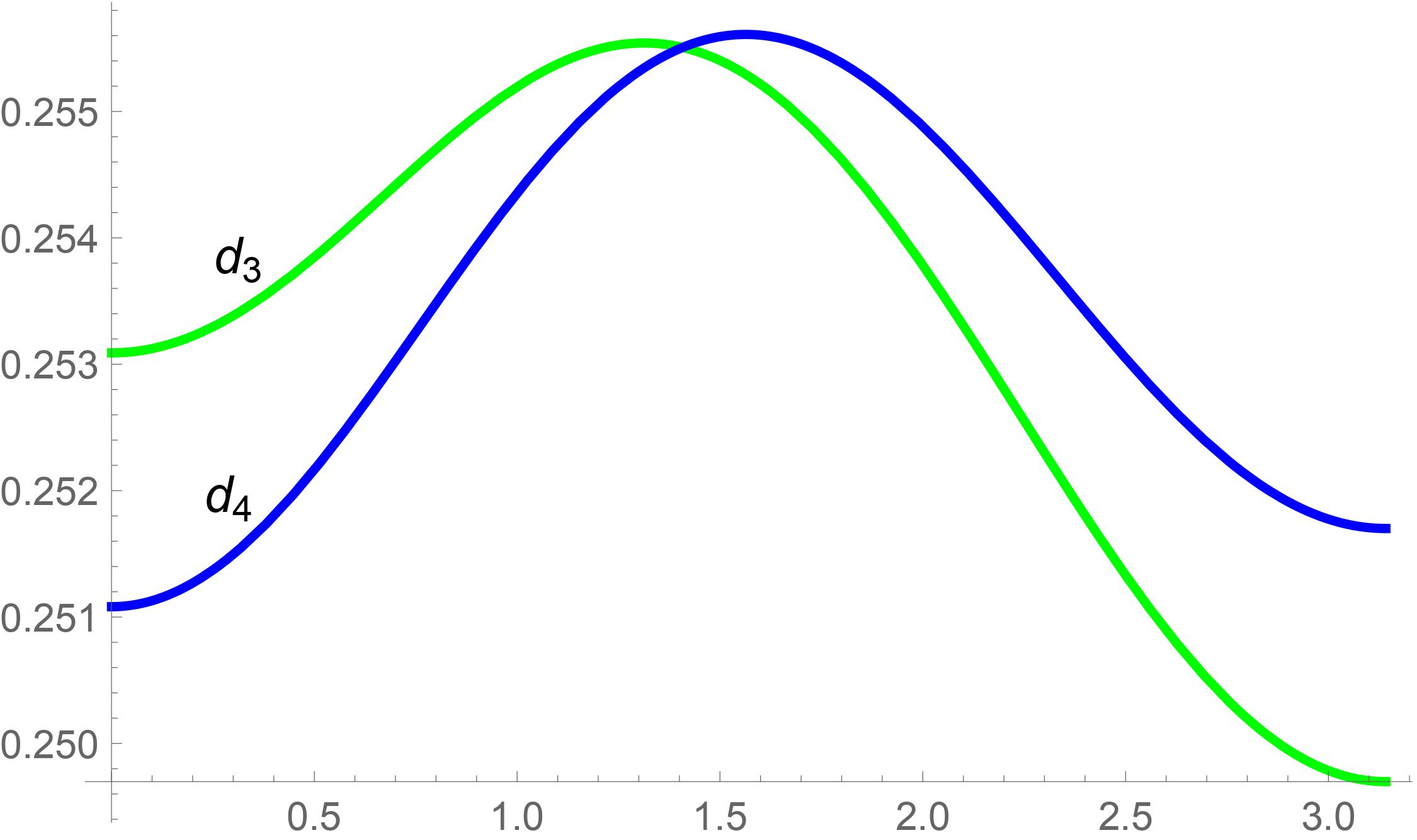}}
    \caption{Graphs of $G_{0.6}(\tau)=d_{1},$ $G_{1.5}(\tau)=d_{3},$ $G_{1.04}(\tau)=d_{3},$ $G_{1.042}(\tau)=d_{4},$ (with $\sigma=1$)  }
    \label{fig1}
\end{figure}

Inclusion results in Lemma \ref{l9}, follows from equation \eqref{e21} and Lemma \ref{l5}.  
\begin{lemma} \label{l9}
For the region $\Omega_{\varrho_{\sigma}}:=\varrho_{\sigma}(\mathbb{D}),$ following inclusion relations hold:
\begin{enumerate}[(i)]
\item $\left\{u:\left|u-( \cosh \sigma + \cos \sigma ) / 2 \right| < (\cosh \sigma - \cos \sigma)/2\right\}\subset\Omega_{\varrho_{\sigma}}.$

\item $\Omega_{\varrho_{\sigma}}\subset \left\{u:\cos \sigma <\operatorname{Re} u < \cosh \sigma \right\}$ and $\Omega_{\varrho_{\sigma}}\subset \left\{u:\cos \sigma < |u| < \cosh \sigma \right\}.$
\item $\Omega_{\varrho_{\sigma}}\subset \left\{u:|\operatorname{Im} u| < l \right\}$ and $\Omega_{\varrho_{\sigma}}\subset \left\{u:|u-(\cosh \sigma+\cos \sigma)/2| < l \right\},$ where $l = |\operatorname{Im} (\cosh(\sigma e^{i t_{0}/2}))|,$ and $t_{0}$ is the root of the equation \[\cos \sigma + \cosh \sigma - 2\cos \left(\sigma \sin \left(t/2\right)\right) \cosh \left(\sigma \cos \left(t/2\right)\right)=0.\]
\end{enumerate}
\end{lemma}

For $\sigma=1,$ Lemma \ref{l5} leads to the following result for the region $\Omega_{\varrho_{1}}=:\Omega_{\varrho}.$ 
\begin{theorem}\label{t14}
The region
  $\Omega_{\varrho}:=\varrho(\mathbb{D})\supset\left\{u\in \mathbb{C}:|u-c|<r_{c}\right\}$
where 
\begin{align}\label{r1}
r_{c}=\left\{\begin{array}{cl}
  c - \cos 1, & \cos 1 < c \leq (\cosh 1 + \cos 1)/2\\
  \cosh 1 - c, & (\cosh 1 + \cos 1)/2 \leq c < \cosh 1.
\end{array}\right.
\end{align}
\end{theorem}

\begin{remark}\label{l12}
Theorem \ref{t14}  ensures that $D_{c}:=|u-c|<r_{c},$ is the maximal disc subscribed in  $\varrho(\mathbb{D}),$ when $c=(\cosh 1 +\cos 1)/2$ and $r_{c}=(\cosh 1 - \cos 1)/2.$ Thus $ D_{c} \subset \varrho(\mathbb{D}).$
\end{remark}

For all the subsequent results, we shall assume $c_{0}:=\cos 1$ and $c_{1}:=\cosh 1.$  
\begin{lemma}\label{t2-10} For the region $\Omega_{\varrho}:= \varrho{(\mathbb{D})},$ we have the following inclusion relations:

\begin{enumerate}[(i)]
\item $\left\{u:\left|u-(c_{0}+c_{1})/ 2 \right| < (c_{1}-c_{0})/2\right\}\subset\Omega_{\varrho}.$ 
\item
 $\Omega_{\varrho} \subset \left\{u:|\arg u| < m \right\},$  where   $m\approx 0.506053$  $\approx (0.322163)$ $\pi/2$ $\approx$ $28.9947^{\circ}.$
\item $\Omega_{\varrho}\subset \left\{u:c_{0} <\operatorname{Re} u < c_{1} \right\}$  and $\Omega_{\varrho}\subset \left\{u:c_{0}< |u| < c_{1} \right\}.$
\item $\Omega_{\varrho}\subset \left\{u:|\operatorname{Im} u| < l \right\}$ and $\Omega_{\varrho}\subset \left\{u:|u-(c_{0}+c_{1})/2| < l \right\},$ where $l = |\operatorname{Im} (\cosh( e^{i t_{0}/2}))|$ and $t_{0}$ is the solution of the equation \[c_{0}+c_{1} - 2\cos \left( \sin \left(t/2\right)\right) \cosh \left( \cos \left(t/2\right)\right)=0 .\]
 \end{enumerate}
\end{lemma} 
\begin{proof}
 We can obtain  $(i)$, $(iii)-(iv)$ from  equations in \eqref{e21}, Remark \ref{l12} and Lemma \ref{l9} (for $\sigma=1$).  
 For part $(ii)$ let  
$\Gamma:=\partial(\varrho(z))=\varrho(e^{i t}),$  $-\pi \leq t\leq\pi,$  
 represents the boundary curve of $\varrho(z).$ Assume that 
 \[\operatorname{Re}{\varrho(e^{i t})}=\cos \left(\sin \left( t /2\right)\right) \cosh \left(\cos \left(t/2\right)\right)=:X(t)\] and 
\[\operatorname{Im}{\varrho(e^{i t})}= \sin \left(\sin \left(t/2\right)\right) \sinh \left(\cos \left(t/2\right)\right)=:Y(t).\] 
 Consider 
 \begin{align*}
 |\arg \varrho(z)| & < \max _{|z|=1 } \hskip 0.1cm |\arg \varrho(z)| = \max _{t\in [-\pi,\pi]}|\arg \varrho(e^{it})| = \max _{t\in [-\pi,\pi]} \tan^{-1} (Y(t)/X(t))\\&
 = \max_{t\in [-\pi,\pi]} \tan^{-1}(\tan \left(\sin \left(t/2\right)\right) \tanh \left(\cos \left(t/2\right)\right))=: m(t)
 \end{align*}

Observe that $\tan^{-1} x$ is a monotonically increasing real valued function. Therefore it is enough to obtain the maximum of $m(t).$ The roots of \begin{align*}
    m'(t)&= 0.5 (\cos (t/2) \tanh (\cos (t/2)) \sec ^2(\sin (t/2))  \\&\quad -\sin (t/2) \tan (\sin (t/2)) \sech^2(\cos (t/2)))=0
\end{align*}
 are  $t_{1}\approx -1.91672$ and $t_{2}\approx 1.91672.$ As $t_{1}<t_{2},$ therefore maximum of  $m(t)$ is attained at $t=t_{2}.$ Hence the inclusion in $(ii)$ follows. 
 \end{proof}

 In Theorem \ref{t13} and Corollary \ref{t12}, we prove inclusion results pertaining to various classes along with the classes $\mathcal{ST}_{p}(\gamma),$ $\mathcal{S}^{*}_{hpl}(s),$ $k-\mathcal{ST}$  and $\mathcal{M}(\beta)$ \cite{R.M. Ali (2008),Kanas & Ebadian,S. Kanas & A. Wisniowska,Uralegaddi(1994)}  defined below:

 \[\mathcal{M}(\beta):=\left\{f \in\mathcal{A}: \frac{zf'(z)}{f(z)} \prec \frac{1+(2\beta-1)z}{1+z}, \beta >1 \right\}, \]

 \[\mathcal{ST}_{p}(\gamma) =\left\{f\in\mathcal{A}:\operatorname{Re}\frac{zf'(z)}{f(z)}+\gamma > \left|\frac{zf'(z)}{f(z)} - \gamma\right|,  \gamma >0 \right\},\]
 
 \[\mathcal{S}^{*}_{hpl}(s) :=\left\{f\in\mathcal{A}: \frac{zf'(z)}{f(z)} \prec (1 - z)^{-s} = e^{-s \log(1-z)},  0<s\leq 1 \right\},\]

\[
k-\mathcal{ST} :=\left\{f\in\mathcal{A}:\operatorname{Re} \frac{zf'(z)}{f(z)}> k \left| \frac{zf'(z)}{f(z)} - 1 \right|,  k\geq 0 \right\}.	\]

\begin{theorem}\label{t13}
Let $f\in \mathcal{S}^{*}_{\varrho_{\sigma}}$ then for each $\sigma\in[-\pi/2,\pi/2]-\{0\},$ following inclusions hold:

\begin{enumerate}[(i)]
\item $\mathcal{S}^{*}_{\varrho_{\sigma}}\subset S^{*}(\zeta),$ where $\zeta=\cos \sigma.$
\item $\mathcal{S}^{*}_{\varrho_{\sigma}}\subset \mathcal{M}(\beta),$ where $\beta=\cosh \sigma.$
\item $\mathcal{S}^{*}_{q_{\kappa}}\subset \mathcal{S}^{*}_{\varrho_{\sigma}},$ whenever $\kappa \leq 1- \cos^{2} \sigma .$
\item $k-\mathcal{ST}\subset \mathcal{S}^{*}_{\varrho_{\sigma}},$ whenever $k \geq \cosh \sigma/(\cosh \sigma - 1).$ 

\item $\mathcal{S}^{*}_{\varrho_{\sigma}}$ $\subset$ $\mathcal{S}_{hpl}^{*}(s) ,$ whenever $ \log(\sec \sigma)/ \log 2 \leq s \leq 1,$ $\sigma\in[-\pi/3,\pi/3]-\{0\}.$

\item $\mathcal{S}^{*}_{\varrho_{\sigma}}\subset \mathcal{S}_{L}^{*}(s),$ whenever $1 - \sqrt{\cos \sigma} \leq s\leq \frac{1}{\sqrt{2}}.$

\end{enumerate}
\end{theorem}
 
 \begin{proof}

Observe that, in equation \eqref{e21}, when $r$ tends to $1^{-},$  sharp bounds on real part and modulus of $\varrho_{\sigma}(z)$ are obtained. Consequently, due to Lemma \ref{l9} the inclusions in $(i)$ and $(ii)$ are true for the class $\mathcal{S}^{*}_{\varrho_{\sigma}}.$ We know that $q_{\kappa}(z)=\sqrt{1+\kappa z}$ where $0<\kappa\leq 1,$ is associated with the region  $|u^{2}-1|<\kappa.$ Therefore part $(iii)$ can be easily established as $q_{\kappa}(\mathbb{D})$ lies in $\Omega_{\varrho_{\sigma}},$ if and only if, $\sqrt{1-\kappa} \geq \cos \sigma,$ which implies $\kappa\leq 1- \cos ^{2} \sigma .$
For part $(iv),$ let $\Gamma_{k}=\left\{u\in\mathbb{C}:\operatorname{Re} u > k |u-1|\right\},$ where $k\geq 0.$ When $k>1,$ the set $\Gamma_{k}$ represents the interior of an ellipse, 
\[\gamma_{k}:=\left\{(x,y):\dfrac{\left(x-x_{1}\right)^{2}}{a_{1}^{2}}+\frac{y^{2}}{b_{1}^{2}}=1\right\},\] 
where $x_{1}=k^{2}/(k^{2}-1),$ $a_{1}=k/(k^{2}-1)$ and $b_{1}=1/\sqrt{k^{2}-1}.$ For $\gamma_{k}$ to lie in $\Omega_{\varrho_{\sigma}}$ we must have $x_{1}+a_{1}\leq \cosh \sigma,$ which gives a sufficient condition for $\gamma_{k}$ to lie in $\Omega_{\varrho_{\sigma}},$ this leads us to the required condition.  From \cite{Kanas & Ebadian} we know that $\operatorname{Re}(1 - z)^{-s} > 2^{-s}.$ Therefore for $(v)$ to hold true $2^{-s}\leq \cos \sigma,$ which gives $\log(\sec \sigma)/ \log 2 \leq s \leq 1,$ provided $-\pi/3\leq \sigma \leq \pi/3.$ Furthermore, it was demonstrated in \cite{Masih n Kanas}, that $L_{S}(\mathbb{D})\supset \{u:|u-1|<1-(1-s)^{2}\},$ where $0< s \leq 1/\sqrt{2}.$ Thus  for  $(vi)$ to hold true we must have $1 - (1-s)^{2}  \geq 1 - \cos \sigma.$ Thus $\mathcal{S}^{*}_{\varrho_{\sigma}}\subset \mathcal{S}^{*}_{L}(s)$ for each $ s \geq 1 -\sqrt{\cos \sigma}.$ 
\end{proof} 
In the following Corollary we prove inclusion results for the class $\mathcal{S}^{*}_\varrho.$ 
\begin{corollary}\label{t12}
 For each function  $f\in \mathcal{S}^{*}_{\varrho}$ the following inclusions hold:
\begin{enumerate}[(i)]
\item $\mathcal{S}^{*}_{\varrho}\subset \mathcal{S}^{*}(c_{0}).$ 
\item $\mathcal{S}^{*}_{\varrho}\subset \mathcal{M}(c_{1}).$ 
\item $\mathcal{S}^{*}_{\varrho}\subset \mathcal{SS}^{*}(\beta),$ where $\beta \approx 0.3222163.$
\item $\mathcal{S}^{*}_{q_{\kappa}}\subset \mathcal{S}^{*}_{\varrho},$ whenever $\kappa \leq 1- c_{0}^{2} .$
\item $k-\mathcal{ST}\subset \mathcal{S}^{*}_{\varrho},$ whenever $k \geq c_{1}/(c_{1} - 1).$ 
\item $\mathcal{S}^{*}_{\varrho}$ $\subset$ $\mathcal{S}_{hpl}^{*}(s),$ whenever $ -\log c_{0}/ \log 2 \leq s \leq 1.$

\item $\mathcal{S}^{*}_{\varrho}\subset \mathcal{S}_{L}^{*}(s),$ whenever $1 - \sqrt{c_{0}} \leq s\leq \frac{1}{\sqrt{2}}.$
\item  $\mathcal{S}^{*}_{\varrho}\subset \mathcal{ST}_{p}(\gamma),$ whenever $\gamma \geq \gamma_{0}\approx 0.0654238.$
\end{enumerate}
\begin{proof}
Clearly parts $(i)-(ii)$ and $(iv)-(vii)$ can be obtained as a result of Theorem \ref{t13} for $\sigma=1.$ Part $(iii)$ is true due to Lemma \ref{t2-10}, for the class $\mathcal{S}^{*}_{\varrho}$ (see Fig. \ref{f2}). For $(viii)$ in order to show  $\mathcal{S}^{*}_{\varrho}\subset \mathcal{ST}_{p}(\gamma),$  we must have $|u - \gamma|-\operatorname{Re} u <\gamma,$ where $u(z)=\cosh\sqrt{z}.$ For $z=e^{it}$ we have
\[H(\tau):=\frac{\sin^{2}(\sin \tau)\sinh ^{2}(\cos \tau)}{4\cos(\sin \tau)\cosh(\cos \tau)}<\gamma,\] where $\tau=t/2.$ Clearly $H'(\tau)$ vanishes on $\{0,\tilde{\tau},\pi/2\},$ with $\tau=\tilde{\tau}\approx 0.832934$ as the only root of the equation
\begin{align*}
   &\tan (\sin \tau) \tanh(\cos \tau) ((\cos \tau (\cos (2 \sin \tau)+3) \sinh (\cos \tau) \sec (\sin \tau))\\&\quad-\sin \tau \sin (\sin \tau) (\cosh (2 \cos \tau)+3) \sech(\cos \tau))=0
\end{align*} in $(0,\pi/2).$ Therefore $\max_{\tau\in [0,\pi/2]}H(\tau)=H(\tilde{\tau})\approx 0.0654238.$ Observe that $\mathcal{ST}_{p}(\gamma_{1})\subset \mathcal{ST}_{p}(\gamma_{2})$ whenever $\gamma_{1}< \gamma_{2}.$ This leads to the required inclusion relation.
\end{proof}
\end{corollary}

\begin{figure}
\begin{tabular}{c}
\includegraphics[scale=0.34]{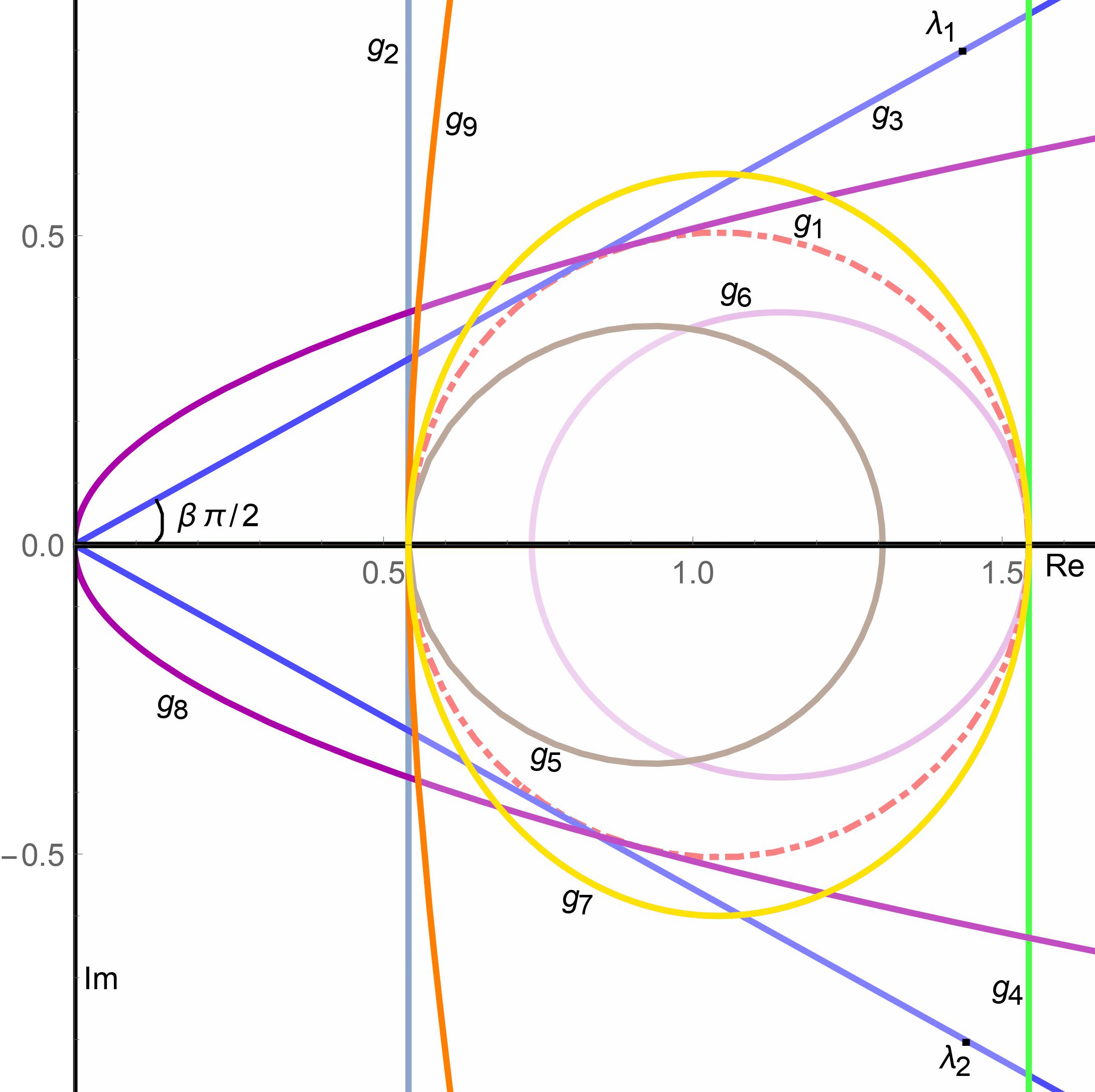}
\end{tabular}
\begin{tabular}{l}
\parbox{0.25\linewidth}{
{\bf \underline{Legend} - }\\
$g_{1}: \varrho(z) =\cosh\sqrt{z} $}\\ \vspace{0.15cm}  
$g_{2}:  \operatorname{Re} u = c_{0}$\\ \vspace{0.15cm} 
$g_{3}: |\arg u| = $ $\beta \pi/2$ \\ \vspace{0.15cm}$\arg \lambda_{1}=-\arg \lambda_{2}=\beta \pi/2$  \\ \vspace{0.15cm} $\beta \approx 0.322163$  \\ \vspace{0.15cm}
$g_{4}: \operatorname{Re} u = c_{1} $\\ \vspace{0.15cm} 
$g_{5}: \sqrt{1 + (1 - c_{0}^{2} ) z}$ \\ \vspace{0.15cm} 
$g_{6}: \operatorname{Re} u = \frac{c_{1}}{c_{1}-1}|u - 1|$\\ \vspace{0.15cm} 
$g_{7}: \frac{(\operatorname{Re} u - \frac{c_{0}+c_{1}}{2})^{2}}{(\frac{c_{1}-c_{0}}{2})^{2}}+\frac{(\operatorname{Im} u)^{2}}{ (c_{2})^{2}},$ \\ \vspace{0.15cm} $c_{2}=0.65$ \\ \vspace{0.15cm} 
$g_{8}: \operatorname{Re} u + \gamma=|u - \gamma|,$ \\  $\gamma \approx 0.0654238$ \\ \vspace{0.15cm} 
$g_{9}: \frac{1}{(1-z)^{s_{0}}},$ $s_{0}= \frac{\log c_{0}^{-1}}{\log 2}$
\end{tabular}
\caption{Inclusion graphs in context of Corollary \ref{t12} associated with  $\varrho(z).$} 
\label{f2}
\end{figure}
\begin{remark}
Fig. \ref{f2}  displays various inclusion relations related to the region $\Omega_{\varrho}:=\Omega_{\varrho_{1}}.$  A vertical ellipse enclosing the region $\Omega_{\varrho}$ is $(x-x_{2})^{2}/a_{2}^{2}+y^{2}/b_{2}^{2}=1,$ where $x_{2}=c_{0}/2,$ $a_{2}=c_{1}/2$ and $c_{2}\geq \max \operatorname{Im}\varrho(z).$  For visual purposes we illustrate this ellipse ($g_{7}$) for $c_{2}=0.65.$  Fig. \ref{f2} depicts the sharpness of inclusion results in Corollary \ref{t12}.
\end{remark}

Let $\mathcal{P}_{n}(\alpha)$ denote the class of functions $p(z)$ of the type $p(z)=1+p_{n}z^{n}+p_{n+1}z^{n+1}+\ldots$ such that $\operatorname{Re} p(z) > \alpha$ $(0\leq \alpha <1).$ Clearly the class $\mathcal{P}_{n}(\alpha)\subset \mathcal{P}_{n}$ and assume $\mathcal{P}_{n}:=\mathcal{P}_{n}(0).$ If a function $p(z)$ of the form $p(z)=1+p_{n}z^{n}+p_{n
+1}z^{n+1}+\cdots$ satisfies $p(z)\prec (1+A z)/( 1 + B z ),$ for   $A\neq B$ and $|B|\leq 1,$ then $p\in \mathcal{P}_{n}[A,B].$ We state a few lemmas in connection with these classes.

\begin{lemma} \label{l11}  \cite{V.Ravi}
If $p\in\mathcal{P}_{n}[A,B],$ then for $|z|=r$
\[\left|p(z)-\frac{1-A B r^{2n}}{1-B^{2}r^{2n}}\right| \leq \frac{|A-B|r^{n}}{1-B^{2}r^{2n}}.\]
Particularly, if $p\in\mathcal{P}_{n}(\alpha),$ then 
\[\left|p(z)-\frac{1+(1-2 \alpha)r^{2n}}{1-r^{2n}}\right|\leq \frac{2(1-\alpha)r^{n}}{1-r^{2n}}.\]	
\end{lemma}

\begin{lemma}\label{l10}\cite{Shah}
If $p\in\mathcal{P}_{n}(\alpha),$ then for $|z|=r$
\[\left|\frac{zp'(z)}{p(z)}\right|\leq \frac{2(1-\alpha)nr^{n}}{(1-r^{n})(1+(1-2\alpha) r^{n})}. \]	
\end{lemma}

\begin{theorem}\label{t8}
Let $p(z)=(1+Az)/(1+Bz),$ where $-1 < B< A \leq 1,$ then $p(z) \prec\cosh   \sqrt{z},$ if and only if 
\begin{align}\label{l6} 
A\leq \left\{\begin{array}{cl}
1-(1-B)c_{0}  & \text{if} \hskip 0.2cm 2(1-AB) \leq (c_{0}+c_{1})(1-B^{2})\\
 (1+B)c_{1} - 1 & \text{if} \hskip 0.2cm 2(1-AB) \geq (c_{0}+c_{1})(1-B^{2}).
 \end{array}\right.
\end{align}
\end{theorem}

\begin{proof}
Lemma \ref{l11} shows that the $p(z)=(1+Az)/(1+Bz),$ maps $\mathbb{D}$ onto the disc
\[\left|p(z)-\frac{1-AB}{1-B^{2}}\right|\leq \frac{A-B}{1-B^{2}}, \quad -1 < B < A \leq 1.\]
By Theorem \ref{t14}, $p(z)\prec\cosh  \sqrt{z}$ if and only if the above disc lies within $\Omega_{\varrho}.$ 
Conditions in \eqref{l6} gives $(1+A)\leq (1+B)c_{1},$ provided $2(1-AB) \geq (c_{0}+c_{1})(1-B^{2})$ holds. Infact $(A-B)/(1-B^{2}) \leq c_{1} - (1 - A B) / (1 - B^{2})$ leads  to $(A-B)/(1-B^{2}) \leq c_{1} - c $ provided $2 c\geq c_{0}+c_{1}$ where  $c=(1 - A B)/(1 - B^{2}).$  Also from \eqref{l6},
$(1-A)\geq (1-B)c_{0}$ whenever $2(1-AB)\leq(c_{0}+c_{1})(1-B^{2}).$ Equivalently, $(A - B)/(1 - B^{2}) \leq c - c_{0}$ whenever $2c\leq c_{0}+c_{1}.$ Thus $p(z)$ lies in $|u-c|<r_{c},$ where $r_{c}$ is given by \eqref{r1}.  
\end{proof}

\begin{corollary}
If conditions on $A,B$ are as given in Theorem \ref{t8}, then $\mathcal{S}^{*}[A,B]\subset\mathcal{S}^{*}_{\varrho}.$
\end{corollary}

\section{Radius Problems}
Radius problems have been an active area of research in geometric function theory. Some of the pioneering work in this direction have been discussed by several authors, see \cite{Aouf n Sokol, Mendiratta n Nagpal, V.Ravi, Hussain & Darus}. For further
development on radius problems of analytic functions, readers may refer
to \cite{Baricz & Obradovic(2013),  Kumar & Sahoo(2021),Ponnusamy & Sahoo(2006),Ponnusamy(2014)}. Motivated by the aforestated work, we derive radius results for the following classes 
\[\mathcal{S}^{*}_{n}(\varrho)=\left\{f\in\mathcal{A}_{n}:\frac{zf'(z)}{f(z)}\prec \cosh  \sqrt{z} =:\varrho(z)\right\},\] 
\[\mathcal{S}^{*}_{n}[A,B]=\left\{f\in\mathcal{A}_{n}:\frac{zf'(z)}{f(z)}\prec \frac{1+A z}{1+B z}\right\}\] and \[\mathcal{M}_{n}(\beta)=\left\{f\in\mathcal{A}_{n}:\frac{zf'(z)}{f(z)}\prec\frac{1+(1 - 2\beta)z}{1-z}, \beta >1 \right\}.\] In the sequel we apply lemmas stated in Section 2, to obtain sharp $\mathcal{S}^{*}_{n}(\varrho)-$radius,  $\mathcal{S}^{*}_{n}[A,B]-$radius and $\mathcal{M}_{n} (\beta)-$radius  for the class $\mathcal{S}^{*}_{\varrho}.$ The following theorem is obtained from Lemma \ref{t2-10} and equations in \eqref{e21}.
\begin{theorem}
The class $\mathcal{S}^{*}_{\varrho}\subset \mathcal{M}(\beta)$ for $|z|<r_{\beta},$ where  
\begin{align*}
r_{\beta}=\left\{\begin{array}{cl}
  r(\beta),& 1<\beta < c_{1} \\
  1,& \beta\geq c_{1}.
 \end{array}\right.
\end{align*}
and $r(\beta)\in(0,1)$ is the smallest root of the equation $\cosh \sqrt{r} =\beta.$ Equality holds when $f(z)=\varphi_{\varrho}(z).$
\end{theorem}
\begin{theorem}
Suppose $f\in\mathcal{S}^{*}_{\varrho},$
then $f(z)$ is starlike of order $\zeta,$ in $|z|<r_{\zeta},$ where  $r_{\zeta}<1$ is the least positive root of the equation $\cos  \sqrt{r}=\zeta.$ This radius result is sharp.
\end{theorem}
\begin{proof}
As $f\in\mathcal{S}^{*}_{\varrho},$ then we have $zf'(z)= f(z) \cosh  \sqrt{w(z)},$ where $w(z)$ is a Schwarz function with $w(0)=0$ such that  for $-\pi\leq t\leq \pi,$ $w(z)=R e^{i t}.$ 
For each $R=|w(z)|\leq |z|=r<1,$ we have $\cos { \sqrt{R}}\geq \cos{ \sqrt{r}},$  and as a result of equations in \eqref{e21}
\[\operatorname{Re}\frac{zf'(z)}{f(z)} \geq \min_{\displaystyle_{|z|=r}} \operatorname{Re}\varrho{(w(z))}=\cos{ \sqrt{r}} \geq \zeta.\] If $s(r,\zeta):=\cos{\sqrt{r}}-\zeta,$ then there exist $r_{\zeta_{0}}<r_{\zeta_{1}}$ such that $s(r_{\zeta_{0}},\zeta)>0$ and $s(r_{\zeta_{1}},\zeta)<0,$ holds. Thus a least positive root $r_{\zeta}$ for the equation $s(r,\zeta)=0,$ will serve the purpose. In particular, at $z_{0}=-r,$ we have 
$\operatorname{Re}(z_{0}\tilde{f}'(z_{0})/\tilde{f}(z_{0}))=\cos{ \sqrt{r}}=\zeta,$ then function $\tilde{f}(z)=\varphi_{\varrho}(z)$ is the extremal function.  
\end{proof}

On replacing $\phi(z)=(1+(1-2\alpha)z)/(1-z)$ in the definition $\mathcal{C}(\phi)$ we get the well-known class of convex functions of order $\alpha$ $(0\leq \alpha < 1),$ denoted by $\mathcal{C}(\alpha).$ For $\alpha=0,$ it reduces to the well-known class of convex functions $\mathcal{C}.$ In Theorem \ref{e23}, we establish radius of convexity of order $\alpha$ for the class $\mathcal{S}^{*}_{\varrho}.$ 

\begin{theorem}\label{e23}
Let $f\in\mathcal{S}^{*}_{\varrho},$ then $f\in\mathcal{C}({\alpha}),$ where $\alpha\in[0,1),$ provided $|z|\leq r_{0},$ where $r_{0}\in [0,1)$ is the least positive root of the equation, 
$2  (1-r^{2})\cos  \sqrt{r}-\sqrt{r}\tan{ \sqrt{r}}=\alpha.$
\end{theorem}
\begin{proof}
As $f\in\mathcal{S}^{*}_{\varrho},$ there exists a Schwarz function $w(z)$  such that  $w(0)=0$ and 
\begin{equation} \label{e57}
\frac{zf'(z)}{f(z)}= \cosh  \sqrt{w(z)}.	
\end{equation}
On logarithmically differentiating \eqref{e57} and applying triangle inequality, we deduce
\begin{align} \label{e9}
\operatorname {Re} \left(1+\frac{zf''(z)}{f'(z)}\right)&  =   \operatorname{Re}\frac{zf'(z)}{f(z)} + \operatorname{Re} \left(\frac{z w'(z)\tanh  \sqrt{w(z)}}{2  \sqrt{w(z)}}\right) \nonumber \\&
\geq \cos  \sqrt{r} -|z||w'(z)|\left|\frac{\tanh  \sqrt{w(z)} }{2 \sqrt{w(z)}}\right| \quad (|z|=r<1) . 
\end{align}
Further Schwarz Pick Lemma, yields 
\begin{equation}\label{e8}
-|z||w'(z)|\left|\frac{\tanh  \sqrt{w(z)}}{ \sqrt{w(z)}}\right|\geq -|z|\frac{1-|w(z)|^{2}}{1-|z|^{2}}\left|\frac{\tanh {\sqrt{w(z)}}}{ \sqrt{w(z)}}\right|.
\end{equation}
Assume $w(z)=Re^{it},$ $t\in[-\pi,\pi]$ where $R\leq r,$   then   
inequality \eqref{e8} yields 
\begin{equation}\label{e11}
\operatorname{Re}\left(\frac{zw'(z)\tanh{\sqrt{w(z)}}}{2 \sqrt{w(z)}}\right)\leq \frac{\sqrt{r}\tan{ \sqrt{r}}}{2  (1-r^{2})}.
\end{equation}
Thus from inequalities \eqref{e9} and \eqref{e11} we conclude that

\begin{equation*}\label{e10}
\operatorname{Re}\left(1+\frac{zf''(z)}{f(z)}\right)\geq \cos  \sqrt{r}  -\frac{\sqrt{r}\tan{ \sqrt{r}}}{2  (1-r^{2})}.
\end{equation*}
 Hence the least positive root of the equation  of $2  (1-r^{2}) \cos \sqrt{r} - \sqrt{r}\tan{\sqrt{r}}=\alpha$ will serve the purpose. 
\end{proof}

\begin{theorem}
For $-1\leq B<A\leq1,$ suppose $f\in\mathcal{S}^{*}_{n}[A,B],$ then the sharp $\mathcal{S}^{*}_{n}(\varrho)-$radius  is given by

\begin{enumerate}[(i)]
\item $\mathcal{R}_{\mathcal{S}^{*}_{n}(\varrho)}(\mathcal{S}^{*}_{n}[A,B])=\min\{1;((1-c_{0})/(A-B c_{0})^{1/n}\}=: \mathcal{R}_{0},$ where $0\leq B <A \leq 1.$
\item $\mathcal{R}_{\mathcal{S}^{*}_{n}(\varrho)}(\mathcal{S}^{*}_{n}[A,B])=
\left\{\begin{array}{cl}
  \mathcal{R}_{0},& \mathcal{R}_{0} \leq \mathcal{R}_{1}\\
  \mathcal{R}_{2},& \mathcal{R}_{0} > \mathcal{R}_{1},
 \end{array}\right.$ where $-1\leq B <0< A \leq 1.$
\end{enumerate}
where 
\[ \mathcal{R}_{1}=\left(\frac{c_{0}-2}{B (c_{0} B - 2 A)}\right)^{1/2 n}, \quad \mathcal{R}_{2}=\min\left\{1;\left(\frac{c_{1}-1}{ A - B c_{ 1}}\right)^{1/n}\right\}.\] 
\end{theorem}

\begin{proof}
As $f\in\mathcal{S}^{*}_{n}[A,B],$ then $p(z)=zf'(z)/f(z)$ lies in the disc $|p(z)-c|<R,$ where \[c=\frac{1-AB r^{2n}}{1-B^{2}r^{2n}} \quad \text{and} \quad R=\frac{(A-B)r^{n}}{1-B^{2}r^{2n}}.\] If $B\geq 0,$ then $c\leq 1.$ For $f(z)$ to lie in $\mathcal{S}^{*}_{n}(\varrho),$ Theorem \ref{t14} and Lemma \ref{l11} yields
   \[\frac{(A-B)r^{n}}{1-B^{2}r^{2n}} \leq \frac{1-AB r^{2n}}{1-B^{2}r^{2n}}-c_{0}.\] The above inequality gives $r\leq R_{0}.$ Equality here holds for $\tilde{f}(z)$ of the form 
 \begin{equation}\label{e13} 
 \tilde{f}(z)= 
\left\{\begin{array}{cl}
  z(1 + B z^{n})^{(A-B)/nB},& B \neq 0\\
  z \exp{(A z^{n}/n)},& B=0.
 \end{array}\right.
 \end{equation}
Further, if $-1 \leq B<0<A\leq 1$ and   $\mathcal{R}_{0} \leq \mathcal{R}_{1},$ then $c\leq (c_{0}+c_{1})/2$ if and only if $r\leq \mathcal{R}_{1}.$ Therefore, for $0\leq r\leq \mathcal{R}_{0},$  we deduce that $c\leq (c_{0}+c_{1})/2.$ Infact due to Theorem \ref{t14} for each $f\in\mathcal{S}^{*}_{n}(\varrho),$ we have $(A-B)r^{n}/(1 - B^{2} r^{2 n})\leq c - c_{0},$ equivalently $r\leq \mathcal{R}_{0}.$  Furthermore assume that $\mathcal{R}_{0} > \mathcal{R}_{1}.$ Then $c\geq (c_{0}+c_{1})/2$ if and only if $r\geq \mathcal{R}_{1}.$ In particular for $ r \geq \mathcal{R}_{0},$ we have $c\geq (c_{0}+c_{1})/2.$ Thus  by Theorem \ref{t14}, for each $f\in\mathcal{S}^{*}_{n}(\varrho),$ the inequality $(A-B)r^{n}/(1 - B^{2} r^{2 n}) \geq c_{1} - c$ is equivalent to $r \leq \mathcal{R}_{2}.$  The function $\tilde{f}(z)$ given in \eqref{e13} works as the extremal function.
\end{proof}

\begin{theorem}
Let $\beta>1,$ then the sharp $\mathcal{S}^{*}_{n}(\varrho)-$radius for the class $\mathcal{M}_{n}(\beta),$  is given by  \[\mathcal{R}_{\mathcal{S}^{*}_{n}(\varrho)}(\mathcal{M}_{n}(\beta))=\left(\frac{1-c_{0}}{2 \beta - (1+c_{0})}\right)^{1/n}.\] 
\begin{proof}
As $f\in\mathcal{M}_{n}(\beta),$ then $z f'(z)/f(z)\prec (1+(1-2 \beta) z)/(1 - z).$ Clearly, for each $\beta>1,$ $(1+(1-2 \beta)r^{2n})/(1-r^{2n}) \leq 1.$   Further by Lemma \ref{l11}, we get
\[\left|\frac{zf'(z)}{f(z)}-\frac{1+(1-2 \beta)r^{2n}}{1-r^{2n}}\right| \leq \frac{2(\beta-1)r^{n}}{1-r^{2n}}.\] On applying Theorem \ref{t14}, we have
\[\frac{2(\beta-1)r^{n}}{1-r^{2n}}\leq \frac{1+(1-2\beta)r^{2n}}{1-r^{2n}}-c_{0}\] or equivalently $r^{2n}((1-2 \beta)+c_{0})-2(\beta-1)r^{n}+1-c_{0}\geq 0,$ which gives $r\leq \mathcal{R}_{\mathcal{S}^{*}_{n}(\varrho)}(\mathcal{M}_{n}(\beta)).$ The required extremal function is
$\tilde{f}(z)=z/(1-z^{n})^{2(1-\beta)/n}.$ 
\end{proof}
\end{theorem}
Recently, Lecko et al. \cite{Lecko & Sim} investigated the expressions $\operatorname{Re}(1-z^{2})f(z)/z>0$ and $\operatorname{Re}(1-z)^{2}f(z)/z>0,$ involving the starlike functions $z/(1-z^{2})$ and $z/(1-z)^{2}.$ In 2019, Cho et al.\cite{Cho & Virender(2019)}  estimated radii constants for classes characterised by the ratio of two analytic functions $f(z)$ and $g(z)$ with certain conditions on $g(z),$ namely $\operatorname{Re}g(z)/z > \alpha$ for $\alpha=0$ or $1/2,$ such that $\operatorname{Re} f(z)/g(z)>0.$   
Motivated by these classes,
here below we define some subclasses of $\mathcal{A}_{n},$ 
\[\mathfrak{F}_{1}(\beta):=\left\{f\in\mathcal{A}_{n}:\left|\frac{f(z)}{g(z)}-1\right|<1 \hskip 0.2cm \text{and} \hskip 0.2cm \operatorname{Re} \frac{g(z)}{z} > \beta , g\in\mathcal{A}_{n}\right\} \quad (\beta \in\left\{0,1/2\right\})\] 
and 
\[\mathfrak{F}_{2}:=\left\{f\in\mathcal{A}_{n}:\left|\frac{f(z)}{g(z)}-1\right|<1 \hskip 0.2cm \text{and} \hskip 0.2cm  g\in\mathcal{A}_{n} \text{ is convex}\right\} .\]
\begin{definition}
Let  $-1\leq A \leq 1$ and $g\in\mathcal{A}_{n},$ then for each $n=1,2,\ldots ,$ $\mathfrak{F}_{3}\subset\mathcal{A}_{n},$ be defined as:
\[\mathfrak{F}_{3}:=\left\{f\in\mathcal{A}_{n}:\operatorname{Re}\dfrac{f(z)}{g(z)}>0 \hskip 0.2cm  \text{and}  \hskip 0.2cm \operatorname{Re}\dfrac{{(1-z^{n})^{(1+A)/n}} g(z) }{z}>0\right\}.\]
\end{definition}

\begin{remark}\label{rem2}
The functions $\tilde{f}(z)=z(1+(1-2 \beta)z^{n})$ and $\tilde{g}(z)=z(1+(1-2 \beta)z^{n})/(1-z^{n})$ defined on $\mathbb{D}$ satisfy $|(\tilde{f}(z)/\tilde{g}(z))-1|=|z|^{n}<1$ and $\operatorname{Re}\tilde{g}(z)/z  =\operatorname{Re} (1+(1-2 \beta)z^{n})/(1-z^{n}) > \beta.$ Therefore $\tilde f\in \mathfrak{F}_{1}(\beta),$ where  $\beta\in\left\{0,1/2\right\}.$ If  $\tilde{f}(z)=z(1+z^{n})/(1 -z^{n})^{1/n}$ and $\tilde{g}(z)=z/(1 - z^{n})^{1/n},$ then $\tilde {f}\in\mathfrak{F}_{2}.$ 
 Similarly when $\tilde{f}(z)=z(1+z^{n})^{2}/(1-z^{n})^{2+(1+A)/n}$  and $\tilde{g}(z)=z(1+z^{n})/(1 - z^{n})^{1+(1+A)/n},$ then  $\tilde{f}\in\mathfrak{F}_{3}.$ Therefore the class $\mathfrak{F}_{3}$ is non-empty. 
\end{remark}

\begin{theorem}
The sharp $\mathcal{S}^{*}_{n}(\varrho)-$ radii for the classes $\mathfrak{F}_{1}(0), \mathfrak{F}_{1}(1/2)$  and $\mathfrak{F}_{2},$ are respectively given by  
\begin{enumerate}[(i)]
    \item $\mathcal{R}_{\mathcal{S}^{*}_{n}(\varrho)}(\mathfrak{F}_{1}(0))=\left(\dfrac{\sqrt{9 n^2-4 (c_{0} - 1) (1+n-c_{0})}-3 n}{2 (1+n-c_{0})}\right)^{1/n}.$
        \item $\mathcal{R}_{\mathcal{S}^{*}_{n}(\varrho)}(\mathfrak{F}_{1}(1/2))=\left(\dfrac{1-c_{0}}{ 2n - (c_{0}-1)}\right)^{1/n}.$
    \item $\mathcal{R}_{\mathcal{S}^{*}_{n}(\varrho)}(\mathfrak{F}_{2})=\left(\dfrac{\sqrt{1+n (n+6)+4 c_{0} (c_{0} -(1+n))}-(1+n)}{2 (n - c_{0})}\right)^{1/n}.$
\end{enumerate}

\end{theorem}

\begin{proof}
Assume  $f(z)/g(z)=p_{1}(z)$ and $g(z)/z=p_{2}(z),$ where $f(z)$ and $g(z)$ are analytic functions in $\mathbb{D}.$ 
\begin{enumerate}[(i)]
\item  As $f\in\mathfrak{F}_{1}(0),$ then $p_{2}\in\mathcal{P}_{n}(0).$ We know  that $|p_{1}(z)-1|<1$ holds if  $\operatorname{Re}(1/p_{1}(z))>1/2$ and vice-versa. Assume $f(z)= z p_{1}(z)p_{2}(z).$ Now using the expressions of $p_{1}(z),$ $p_{2}(z)$ and  by applying Theorem \ref{t14} and Lemma \ref{l10} we have \[\left|\frac{zf'(z)}{f(z)}-1\right|=\left|\frac{z p_{2}'(z)}{p_{2}(z)} - \frac{z p_{1}'(z)}{p_{1}(z)}\right| \leq \frac{(3+r^{n})n r^{n}}{1- r^{2n}}\leq 1- c_{0}.\] The above inequality leads to  $r^{2n}(n +1 - c_{0})+3 n r^{n} -1 + c_{0} \leq 0,$ 
provided $r\leq \mathcal{R}_{\mathcal{S}^{*}_{n}(\varrho)}(\mathfrak{F}_{1}(0)).$ The functions $\tilde{f}(z)=z(1+z^{n})/(1-z^{n})^{2}$  and $\tilde{g}(z)=z (1+z^{n})/(1-z^{n})$  at $z_{0}=\mathcal{R}_{\mathcal{S}^{*}_{n}(\varrho}(\mathfrak{F}_{1}(0))e^{i\pi/n}$ gives 
\[\frac{z_{0}\tilde{f}'(z_{0})}{\tilde{f}(z_{0})}-1 = \frac{(3+z_{0}^{n})n z_{0}^{n}}{1 - z_{0}^{2n}}=1-c_{0}.\] Thus $\tilde{f}$ is the extremal function.
\item  As $f\in\mathcal{R}_{\mathcal{S}^{*}_{n}(\varrho)}(\mathfrak{F}_{1}(1/2)),$ then $1/p_{1},p_{2}\in\mathcal{P}_{n}(1/2).$ Proceeding as in (i), on applying  Theorem \ref{t14} and Lemma \ref{l10} we get 
\[\left|\frac{z f'(z)}{f(z)} - 1  \right| \leq \frac{2 n r^{n}}{1 - r^{n}} \leq 1 - c_{0}.\] This holds true whenever $r\leq \mathcal{R}_{\mathcal{S}^{*}_{n}(\varrho)}(\mathfrak{F}_{1}(1/2)).$ For sharpness, consider 
$\tilde{f}(z) = z $ and  $\tilde{g}(z) = z/(1-z^{n}),$ then at $z_{0}=\mathcal{R}_{\mathcal{S}^{*}_{n}(\varrho)}(\mathfrak{F}_{1}(1/2))e^{i\pi/n},$ we get
\[\frac{z_{0}\tilde{f}'(z_{0})}{\tilde{f}(z_{0})}-1 = \frac{2 n z_{0}^{n}}{1 - z_{0}^{n}} = 1 - c_{0}.\]
\item Let $f(z)/g(z)=p(z)$ be a function defined in $\mathbb{D}.$ As $f\in\mathfrak{F}_{2},$ then $|1/p(z)-1|<1$ if and only if $\operatorname{Re} p(z)>1/2.$ As $g\in\mathcal{A}_{n}$ is convex, then due to Marx-Strohh$\ddot{a}$cker theorem, $g\in\mathcal{S}_{n}^{*}(1/2),$ $(\mathcal{S}^{*}_{n}(1/2)=\{f\in\mathcal{A}_{n}:\operatorname{Re}zf'(z)/f(z)>1/2\})$. Therefore due to Lemma \ref{l11},  \[\left|\frac{zg'(z)}{g(z)}-\frac{1}{1-r^{2 n}}\right| \leq \frac{r^{n}}{1-r^{2n}}.\] On logarithmically differentiating $f(z)$ and applying Theorem \ref{t14}, we get
\begin{align*}
    \left|\frac{zf'(z)}{f(z)}-\frac{1}{1-r^{2n}}\right|&=\left|\frac{z g'(z)}{g(z)} - \frac{z p'(z)}{p(z)} -\frac{1}{1-r^{2n}} \right| \\& \leq \frac{n r^{2n} + (1+n)r^{n}}{1 - r^{2 n}}  \leq \frac{1}{1-r^{2n}} - c_{0},
\end{align*} which leads to $r^{2n}(n - c_{0}) +r^{n} (1+n) -1+c_{0} \leq 0,$ provided $r \leq \mathcal{R}_{\mathcal{S}^{*}_{n}(\varrho)}(\mathfrak{F}_{2}).$ The functions $\tilde{f}(z)=z(1+z^{n})/(1 -z^{n})^{1/n}$ and $\tilde{g}(z)=z/(1 - z^{n})^{1/n}$ 
at $z_{0}=\mathcal{R}_{\mathcal{S}^{*}_{n}(\varrho)}(\mathfrak{F}_{2})e^{i\pi/n}$ gives  $|z_{0}\tilde{f}'(z_{0})/\tilde{f}(z_{0})| = c_{0}.$ Hence the result is sharp.
\end{enumerate}
\end{proof}

\begin{theorem}
Let $r\in[0,1),$ then the sharp $\mathcal{S}^{*}_{n}(\varrho)-$radius for the class  $\mathfrak{F}_{3}$ is given by  
\[ \mathcal{R}_{\mathcal{S}^{*}_{n}(\varrho)}(\mathfrak{F}_{3})
=\left\{\begin{array}{cl}
  \mathcal{R}_{0},&   r \leq  \mathcal{R}_{0}, \\
  \mathcal{R}_{1},&  r \geq \mathcal{R}_{0},
 \end{array}\right. \]
   where 
  \begin{align*}
 \mathcal{R}_{0}& =  
  \begin{cases} 
    \left(\dfrac{1+A+4 n +\sqrt{(1+A+4 n)^2-4 (1-c_{0}) (A + c_{0})}}{2 (A+c_{0})}\right)^{1/n} &\text{if }-1\leq A <-c_{0},  \\ & \\
      \left(\dfrac{1 - c_{0}}{1+4n - c_{0}}\right)^{1/n} & \text{if }   A = - c_{0}, \\& \\
      \left(\dfrac{1+A+4 n - \sqrt{(1 + A + 4 n)^2-4 (1-c_{0}) (A + c_{0})}}{2 (A+c_{0})}\right)^{1/n}  & \text{if } - c_{0} < A \leq 1,
 \end{cases}
\end{align*}
and  \[ \mathcal{R}_{1} = \left(\frac{\sqrt{(1 + A + 4 n)^{2} + 4 (A 
 +c_{1})(c_{1} - 1)}-(1 + A + 4 n )}{2 (A + c_{1})}\right)^{1/n}.\]

\begin{proof}
Let $f\in \mathfrak{F}_{3},$ then  $\operatorname{Re}f(z)/g(z)>0$ and $\operatorname{Re}((1-z^{n})^{(1+A)/n}g(z)/z)>0,$ where $g\in\mathcal{A}_{n}.$ Define $g(z)/f(z)=p_{1}(z)$ and $(1-z^{n})^{(1+A)/n}g(z)/z=p_{2}(z),$ where $p_{1}(z)$ and $p_{2}(z)$ are analytic in $\mathbb{D}.$  Since $A<1,$ then for $|z|=r<1,$ the inequality $(1 + A r^{2n})\geq 1-r^{2n},$ holds true.
Further on logarithmically differentiating $zp_{1}(z)p_{2}(z)(1-z^{n})^{-(1+A)/n}=f(z),$  we get 
\[\frac{zf'(z)}{f(z)}=\frac{1+Az^{n}}{1-z^{n}}+\frac{z p_{1}'(z)}{p_{1}(z)}+\frac{zp_{2}'(z)}{p_{2}(z)}.\]
Due to Lemmas \ref{l11} - \ref{l10}, for $|z|=r,$ we infer 

\begin{equation}\label{e61}
\left|\frac{zf'(z)}{f(z)} - \frac{1 + A r^{2n}}{1 - r^{2n}} \right| \leq \frac{4 n r^{n}}{1 - r^{2n}}+\frac{(1 + A)r^{n}}{1-r^{2n}}. 	
\end{equation}
Assume $c=(1+A r^{2n})/(1 - r^{2n}).$ Then $c\leq (c_{0}+c_{1})/2$ leads to $r\leq \mathcal{R}$ and vice-versa, where $\mathcal{R}=\left((c_{0} -2)/(2 A +c_{0})\right)^{1/2n}.$  Algebraically, for each $n=1,2,3, \ldots,$ it can be observed that, for the given range of $A,$ we have $\mathcal{R}_{0}<\mathcal{R}_{1}<\mathcal{R}.$  In particular, if $r\leq \mathcal{R}_{0},$ then $c\leq (c_{0}+c_{1})/2.$ Further due to Theorem \ref{t14}, inequality \eqref{e61} gives 
\[\frac{4 n r^{n}}{1 - r^{2n}}+\frac{(1 + A)r^{n}}{1-r^{2n}}\leq \frac{1 + A r^{2n}}{1 - r^{2n}} - c_{0},\]
whenever $r\leq \mathcal{R}_{0}.$ Moreover if $c \geq (c_{0}+c_{1})/2,$ then $r\geq \mathcal{R}_{0}.$ Infact, when $r\geq \mathcal{R}_{0},$ then we have $c\geq (c_{0}+c_{1})/2.$ Now  inequality \eqref{e61} together with Theorem \ref{t14} yields
\[\frac{4 n r^{n}}{1 - r^{2n}}+\frac{(1 + A)r^{n}}{1-r^{2n}}\leq c_{1} - \frac{1 + A r^{2n}}{1 - r^{2n}},\]
provided $r\leq \mathcal{R}_{1}.$ Thus the following functions, mentioned in Remark \ref{rem2}
\begin{align*}
\tilde{f}(z)=\frac{z(1+z^{n})^{2}}{(1-z^{n})^{2+(1+A)/n}} \quad \text{and} \quad \tilde{g}(z)= \frac{z(1+z^{n})}{(1-z^{n})^{1+(1+A)/n}},
\end{align*}
serve as the extremal function for both the cases. 
\end{proof}
\end{theorem}

\section{Certain estimates for the class  $\mathcal{S}^{*}_{\varrho}$}
In this section certain sufficient conditions for the class $\mathcal{S}^{*}_{\varrho}$ are established. 
\begin{theorem}\label{t4}
Let $f\in\mathcal{A}$ , then $f\in\mathcal{S}^{*}_{\varrho}$ if and only if 
\begin{equation}\label{e62}
\frac{1}{z}\left(f(z)*\frac{z- k z^{2}}{(1 - z)^{2}}\right) \neq 0
 \end{equation}
 	where $k=\cosh  e^{it/2}/(\cosh  e^{it/2} - 1)$ for $t\in [-\pi,\pi].	$
Moreover, $f\in\mathcal{S}^{*}_{\varrho}$ if and only if 
\begin{equation}\label{t5}
1-\sum_{n=2}^{\infty}\frac{(n - \cosh  e^{it/2})a_{n}}{ \cosh  e^{it/2} - 1 } z^{n-1} \neq 0.
\end{equation}	
\end{theorem}
\begin{proof}
Since $f\in\mathcal{S}^{*}_{\varrho},$  then $zf'(z)/f(z) = \cosh  \sqrt{w(z)},$ where $w(z)$ is a Schwarz function with $w(0)=0.$ Equivalently for $w(z)=e^{it},$ $-\pi \leq t \leq \pi,$ we have
 \[\frac{zf'(z)}{f(z)} \neq \cosh  e^{it/2} \Leftrightarrow zf'(z) - (\cosh  e^{it/2})f(z) \neq 0 \quad \text{for}  \quad t\in[-\pi,\pi],\]
Eventually it leads to $zf'(z) - k (zf'(z)-f(z)) \neq 0.$ Thus through simple computations \eqref{e62} can be established. The condition in \eqref{t5} can be deduced using \eqref{e62}. 
\end{proof}

\begin{corollary}
Let $f\in\mathcal{A}$ satisfy the following 
\begin{equation}\label{e63}
\sum_{n=2}^{\infty} \left|\frac{n- \cosh  e^{it/2}}{\cosh  e^{it/2} - 1}\right||a_{n}| < 1, 
	\end{equation}
then $f\in\mathcal{S}^{*}_{\varrho}.$
\end{corollary}
\begin{proof}
	 Consider the following inequality with  $k=\cosh  e^{it/2}/(\cosh  e^{it/2}-1),$
	\[\left|1-\sum_{n=2}^{\infty} (n(k-1)-k)a_{n}z^{n-1}\right| \geq 1 - \sum_{n=2}^{\infty}|n (k-1) -k||a_{n}|.\]
	Thus from \eqref{e63} we establish  \[\left|1-\sum_{n=2}^{\infty} (n(k-1)-k)a_{n}z^{n-1}\right|> 0,\]
	Hence due to Theorem \ref{t4} we conclude that $f\in\mathcal{S}^{*}_{\varrho}.$
\end{proof}

\begin{theorem}
Let $f\in\mathcal{S}^{*}_{\varrho}$  then the following inequality holds
\begin{equation*}
c_{1}^{2} -1 \geq \sum_{k=2}^{\infty}(k^{2}-{c_{1}}^{2})|{a_{k}}|^{2}.
\end{equation*}
\end{theorem}

\begin{proof}
Since $f\in\mathcal{S}^{*}_{\varrho},$ then $zf'(z)=\cosh(  \sqrt{w(z)})f(z),$ for a Schwarz function $w(z)$ with $w(0)=0.$ For $0\leq |z|=r <1,$ we get the following 
\begin{align}
2 \pi \sum_{k=1}^{\infty} k^{2}|a_{k}|^{2}r^{2k}&=\int_{0}^{2 \pi}\left|r e^{i\theta}f'(r e^{i \theta})\right|^{2}d\theta \nonumber \\&
=\int_{0}^{2\pi}\left|\cosh\left(\sqrt{w(r e^{i \theta})}\right)f(r e^{i \theta})\right|^{2} d\theta\nonumber \\&
\leq \int_{0}^{2\pi}\cosh^{2}\left(\sqrt{|w(re^{i\theta})|}\right)|f(r e^{i \theta})|^{2} d\theta\nonumber \\&
\leq \int_{0}^{2\pi} (\cosh ^{2} r) |f(r e^{i \theta})|^{2} d \theta \nonumber \\ &
= 2 \pi (\cosh^{2} r) \sum_{k=1}^{\infty}|a_{k}|^{2}r^{2k}.
\end{align}	
Thus when $r$ tends to $1^{-},$ we at once obtain the required inequality.
\end{proof}

\begin{example}
Let $f\in\mathcal{A},$ then following functions are members of $\mathcal{S}^{*}_{\varrho}.$
\begin{enumerate}[(i)]
\item $f(z)=z+a_{n}z^{n} \in \mathcal{S}^{*}_{\varrho},$ provided $|a_{n}|\leq (1 - c_{0})/(n-c_{0}),$ $n\in\mathbb{N}-\{1\}.$  
\item $f(z)=z/(1 - A z)^{2}\in \mathcal{S}^{*}_{\varrho},$ provided $|A|\leq (c_{1} - 1 )/(c_{1} + 1).$
\item $f(z)=z/(1 - A z)\in \mathcal{S}^{*}_{\varrho},$ provided $|A|\leq (c_{1} - 1)/c_{1}.$
\item   $f(z)=z e^{A z}\in \mathcal{S}^{*}_{\varrho},$ provided $|A|\leq 1 - c_{0}.$
\end{enumerate}
\end{example}

\begin{proof}
For part (i) we require that $zf'(z)/f(z)=(1+n a_{n} z^{n-1})/(1 + a_{n} z^{n-1})$ must lie in the disc $\left\{u:|u-c|<r_{c}\right\}\subset\varrho(\mathbb{D}),$ centered at $c,$ where $r_{c}$ is defined in \eqref{r1}. It is a known fact that the function $f(z)=z+a_{n}z^{n}$ is univalent if and only if $|a_{n}|\leq 1/n.$ Thus $c=(1-n |a_{n}|^{2})/(1 - |a_{n}|^{2})\leq 1.$ If $u=(1+n a_{n} z^{n-1})/(1 + a_{n} z^{n-1})$  and $r_{c}=(1-n|a_{n}|^{2})/(1-|a_{n}|^{2})-c_{0},$ then due to Theorem \ref{t14}, 
\[\frac{(n-1)|a_{n}|}{1 - |a_{n}|^{2}}\leq \frac{1-n |a_{n}|^{2}}{1 - |a_{n}|^{2}} - c_{0}.\] The proofs of (ii)-(iv) are much akin to (i), therefore it is skipped. 
\end{proof}


\begin{thebibliography}{99}

\bibitem{R.M. Ali (2008)}  Ali, R.M., Subramanian, K.G., Ravichandran V., Ahuja, O.P.: Neighborhoods of starlike and convex functions associated with parabola. J. Inequal. Appl. {\bf 346279}, 9 (2008) 


\bibitem{Hussain} Alotaibi A.,  Arif M., Alghamdi, M.A., Hussain, S.: Starlikness associated with cosine hyperbolic function. Mathematics {\bf 8}(7), 1118(2020). https://doi.org/10.3390/math8071118



\bibitem{Aouf n Sokol} Aouf, M. K., Dziok J., Sok\'{o}\l, J.: On a subclass of strongly starlike functions. Appl. Math. Lett. {\bf 24}(1), 27--32(2011)



\bibitem{Raza} Bano, K., Raza, M.: Starlike functions associated with cosine functions. Bull. Iranian Math. Soc. {\bf 47}(5), 1513--1532(2021)


\bibitem{Baricz & Obradovic(2013)} Baricz, \'{A}., Obradovi\'{c}, M., Ponnusamy, S.: The radius of univalence of the reciprocal of a product of two analytic functions. J. Anal. {\bf 21}, 1–19(2013)


\bibitem{Cho & Virender(2019)} Cho, N.E., Kumar, V., Kumar, S.S., Ravichandran, V.: Radius problems for starlike functions associated with the sine function. Bull. Iranian Math. Soc. {\bf 45}(1), 213--232(2019)












\bibitem{Janowski} Janowski, W.: Some extremal problems for certain families of analytic functions. I. Ann. Polon. Math. {\bf 28}(3), 297--326(1973)



\bibitem{Kanas & Ebadian} Kanas, S., Masih, V. S., Ebadian, A.: Relations of a planar domains bounded by hyperbolas with families of holomorphic functions. J. Inequal. Appl, {\bf 2019}(1), 1--14(2019)



\bibitem{S. Kanas & A. Wisniowska} Kanas, S.,  Wi\'{s}niowska, A.: Conic domains and starlike functions. Rev. Roumaine Math. Pures Appl. {\bf 45}(4), 647--658(2000)










\bibitem{Kumar & Sahoo(2021)}
Kumar, S., Sahoo, S. K.: Radius of convexity for integral operators involving Hornich operations. J. Math. Anal. Appl. {\bf 502}(2),  Paper No. 125265, 21 pp, (2021)



\bibitem{Lecko & Sim} Lecko, A., Sim, Y. J.: Coefficient problems in the subclasses of close-to-star functions. Results Math. {\bf 74}(3), 104(2019)


\bibitem{Ma & Minda}  Ma, W. C., Minda, D.: A unified treatment of some special classes of univalent functions. In:  Proceedings of the Conference on Complex Analysis (Tianjin, 1992), pp. 157--169, Conf. Proc. Lecture Notes Anal., I, Int. Press, Cambridge, MA

\bibitem{Masih n Kanas} Masih, V. S., Kanas, S.: Subclasses of starlike and convex functions associated with the Lima\c{c}on Domain. Symmetry {\bf 12}, 942(2020)

\bibitem{Mendiratta n Nagpal} Mendiratta, R.,  Nagpal, S., Ravichandran, V.: On a subclass of strongly starlike functions associated with exponential function. Bull. Malays. Math. Sci. Soc. {\bf 38}(1), 365--386(2015)





\bibitem{Mundalia & Sivaprasad(2020)} Mundalia, M., Shanmugam, S. K.: Coefficient bounds for a unified class of holomorphic functions. In  Mathematical analysis. I. Approximation theory, pp. 197--210, Springer Proc. Math. Stat., 306, Springer, Singapore.




\bibitem{Nezhmetdinov & Ponnusamy(2005)} Nezhmetdinov, I. R., Ponnusamy, S.: New coefficient conditions for the starlikeness of analytic functions and their applications. Houston J. Math. {\bf 31}(2), 587–604(2005)




\bibitem{Ponnusamy & Sahoo(2006)} Ponnusamy, S., Sahoo, S. K.: Study of some subclasses of univalent functions and their radius properties. Kodai Math. J. {\bf 29}(3), 391–405(2006)

\bibitem{Ponnusamy(2016)} Ponnusamy, S, Sahoo, S. K., Sharma, N. L.: Maximal area integral problem for certain class of univalent analytic functions. Mediterr. J. Math. {\bf 13}(2), 607–623(2016)


\bibitem{Ponnusamy(2014)} Ponnusamy, S., Sahoo, S. K., Sugawa, T.: Radius problems associated with pre-Schwarzian and Schwarzian derivatives. Analysis (Berlin) {\bf 34}(2), 163–171(2014)


\bibitem{Ponnusamy(2020)} Ponnusamy, S., Sharma, N. L., Wirths, K.-J.: Logarithmic coefficients problems in families related to starlike and convex functions. J. Aust. Math. Soc. {\bf 109}(2), 230–249(2020)


\bibitem{Raina} Raina, R. K.,  Sok\'{o}\l , J.: Some properties related to a certain class of starlike functions. C. R. Math. Acad. Sci. Paris {\bf 353}(11), 973--978(2015)



\bibitem{V.Ravi} Ravichandran, V., R\o nning F., Shanmugam, T. N.:Radius of convexity and radius of starlikeness for some classes of analytic functions. Complex Var. Theory Appl. {\bf 33}(1-4), 265--280(1997)





\bibitem{Robertson(1936)} Robertson, M.I.S.: On the theory of univalent functions. Ann. Math. {\bf 37}(2), 374--408(1936) 






\bibitem{Hussain & Darus} Saliu, A., Noor, K. I.,  Hussain, S., Darus, M.: Some results for the family of univalent functions related with lima\c{c}on domain. AIMS Math. {\bf 6}(4), 3410--3431(2021)

\bibitem{Shah} Shah, G. M.: On the univalence of some analytic functions. Pac. J. Math. {\bf 43}, 239--250(1972)






\bibitem{Sokol n Stankiewicz} Sok\'{o}\l , J.,   Stankiewicz, J.: Radius of convexity of some subclasses of strongly starlike functions. Zeszyty Nauk. Politech. Rzeszowskiej Mat. {\bf 19}, 101--105(1996).


\bibitem{Stankiewicz} Stankiewicz, J.: Quelques probl\`emes extr\'{e}maux dans les classes des fonctions $\alpha $-angulairement \'{e}toil\'{e}es, Ann. Univ. Mariae Curie-Sk\l odowska Sect. A {\bf 20}(1966), 59--75(1971).


\bibitem{Uralegaddi(1994)} Uralegaddi, B. A., Ganigi, M. D., Sarangi, S. M.: Univalent functions with positive coefficients. Tamkang J. Math. {\bf 25}(3), 225--230(1994)




	\end{thebibliography}
\end{document}